\documentclass[10pt]{article}
\usepackage{graphicx}
\usepackage{amsmath}
\usepackage{amssymb}
\newtheorem{theorem}{Theorem}[section]
\newtheorem{lemma}[theorem]{Lemma}

\theoremstyle{definition}
\newtheorem{remark}[theorem]{Remark}

\renewcommand{\theequation}{\thesection.\arabic{equation}}
\numberwithin{equation}{section}

\begin{document}
\setcounter{page}{1}
\setcounter{firstpage}{1}
\setcounter{lastpage}{11}
\renewcommand{\currentvolume}{56}
\renewcommand{\currentyear}{2011}
\renewcommand{\currentissue}{2}
\title[A characterization theorem and its applications for]{A characterization theorem  and its \\applications for $d$-orthogonality
of\\ Sheffer polynomial sets}
\author{Serhan Varma}
\address{Ankara University Faculty of Science, Department of Mathematics,
Tando\u{g}an TR-06100, Ankara, Turkey.}
\email{svarma@science.ankara.edu.tr}

\begin{abstract}
The purpose of this paper is to find the characterization of the Sheffer
polynomial sets satisfying the $d$-orthogonality conditions. The generating
function form of these polynomial sets is given in Theorem 2.2. As
applications of the Theorem 2.2, we revisit the $d$-orthogonal polynomial sets
exist in the literature and discover new $d$-orthogonal polynomial sets.
Moreover, we obtain the $d$-dimensional functional vector ensuring the $d$%
-orthogonality of these new polynomial sets.
\end{abstract}
\maketitle

\section{Introduction}
Recently, the generalization of orthogonal polynomials called "$d$%
-orthogonal polynomials" have attracted so much attention from many authors.
The well-known properties of orthogonal polynomials such as recurrence
relations, Favard theorem, generating function relations and differential
equations have found correspondence in this new notion. New polynomial sets
which contain classical orthogonal polynomials have been created so far. Let
us give a brief summary of d-orthogonal polynomials.

Let $\mathcal{P}$ be the vector space of polynomials with complex
coefficients and $\mathcal{P}^{^{\prime }}$be the vector space of all linear
functionals on $\mathcal{P}$ called the algebraic dual. $\left\langle
u,f\right\rangle $ is the representation of the effect of any linear
functional $u\in \mathcal{P}^{^{\prime }}$ to the polynomial $f\in \mathcal{P%
}$. Let $\left\{ P_{n}\right\} _{n\geq 0}$ be a polynomial set ($\deg \left(
P_{n}\right) =n$ for all non-negative integer $n$), and corresponding dual
sequence $\left( u_{n}\right) _{n\geq 0}$ for polynomials taken from this
set can be given by
\begin{equation*}
\left\langle u_{n},P_{k}\right\rangle =\delta _{nk}\text{ \ \ \ , \ \ \ }%
n,k=0,1,...,\text{\ }
\end{equation*}%
where $\delta _{nk}$ is the Kronecker delta.

A polynomial set $\left\{ P_{n}\right\} _{n\geq 0}$ in $\mathcal{P}$ is said
to be $d$-orthogonal polynomial set with respect to the $d$-dimensional
functional vector $\Gamma =^{t}\left( u_{0},u_{1},...,u_{d-1}\right) $ if
the following orthogonality conditions are hold
\begin{equation}
\left\{
\begin{array}{cc}
\left\langle u_{k},P_{n}P_{m}\right\rangle =0\text{ ,} & m>nd+k\text{ ,} \\
\left\langle u_{k},P_{n}P_{nd+k}\right\rangle \neq 0\text{ ,} & n\geq 0\text{
,}%
\end{array}%
\right.  \label{1}
\end{equation}%
where $n\in
%TCIMACRO{\U{2115} }%
%BeginExpansion
\mathbb{N}
%EndExpansion
_{0}=\left\{ 0,1,2,...\right\} $, $d$ is a positive integer and $k\in
\left\{ 0,1,...,d-1\right\} $ $($see $\left[ 1-2\right] )$. Characterization
of these polynomials by recurrence relations and Favard type theorem was
also given in $\left[ 1-2\right] $. A polynomial set $\left\{ P_{n}\right\}
_{n\geq 0}$ is a $d$-orthogonal polynomial set if and only if it fulfills a $%
\left( d+1\right) $-order recurrence relation of the type%
\begin{equation}
xP_{n}\left( x\right) =\sum\limits_{k=0}^{d+1}\alpha _{k,d}\left( n\right)
P_{n-d+k}\left( x\right) \text{ \ \ \ , \ \ }n\in
%TCIMACRO{\U{2115} }%
%BeginExpansion
\mathbb{N}
%EndExpansion
_{0}\text{ \ ,}  \label{2}
\end{equation}%
with the regularity conditions $\alpha _{d+1,d}\left( n\right) \alpha
_{0,d}\left( n\right) \neq 0$ , $n\geq d$ and by convention $P_{-n}\left(
x\right) =0$ , $n\geq 1$. Taking $d=1$ in $\left( \ref{1}\right) $ and $%
\left( \ref{2}\right) $ leads us to the celebrated notion of orthogonal
polynomials $\left( \text{see \cite{3}}\right) $.

The recurrence relation of order $d+1$ has been the main reason of deriving
many $d$-orthogonal polynomial sets as an extension of known ones in
orthogonal polynomials. Classical orthogonal polynomials such as Laguerre,
Hermite and Jacobi polynomials, discrete orthogonal polynomials like
Charlier, Meixner polynomials and so on were extended to the $d-$%
orthogonality notion and many basic properties linking with these
polynomials were stated by various authors $\left( \left[ 4-23\right]
\right) $. Especially, in \cite{14}, the authors described a useful method
for checking whether if a given polynomial set $\left\{ P_{n}\right\}
_{n\geq 0}$ is $d$-orthogonal or not. This method will be described and used
in the sequel.

A polynomial set $\left\{ P_{n}\right\} _{n\geq 0}$ is called Sheffer
polynomial set if and only if it has the generating function of the form
\begin{equation}
A\left( t\right) e^{xH\left( t\right) }=\sum\limits_{n=0}^{\infty
}P_{n}\left( x\right) \frac{t^{n}}{n!}\text{ \ }  \label{3}
\end{equation}%
where $A\left( t\right) $ and $H\left( t\right) $ have the power series
expansions as following
\begin{equation*}
A\left( t\right) =\sum\limits_{k=0}^{\infty }a_{k}t^{k}\text{ \ , \ }%
a_{0}\neq 0\text{ , }H\left( t\right) =\sum\limits_{k=0}^{\infty
}h_{k}t^{k+1}\text{ \ , \ }h_{0}\neq 0\text{ .}
\end{equation*}%
This means that $A\left( t\right) $ is invertible and $H\left( t\right) $
has a compositional inverse. There are numerous polynomial sets belong to
the class of Sheffer polynomials $\left( \text{see }\left[ 24-25\right]
\right) $. Note that, for $H\left( t\right) =t$, we meet the definition of
Appell polynomial sets \cite{25} from the aspect of generating functions.
That is to say, Appell polynomials can be defined by generating function of
the type%
\begin{equation*}
A\left( t\right) e^{xt}=\sum\limits_{n=0}^{\infty }P_{n}\left( x\right)
\frac{t^{n}}{n!}\text{ \ \ }
\end{equation*}%
with $A\left( t\right) =\sum\limits_{k=0}^{\infty }a_{k}t^{k}$ $\left(
a_{0}\neq 0\right) $. In this contribution, our aim is to find the exact
form of $d$-orthogonal polynomial sets which are at the same time Sheffer
polynomial set. Then, we try to derive new $d$-orthogonal polynomial sets
and find some of them's $d$-dimensional functional vector for which promises
the $d$-orthogonality. Also, we revisit some known $d$-orthogonal polynomial
sets exist in the literature.
\section{Main Results}

Characterization problems related to Sheffer polynomial set have a deep
history. Both Meixner \cite{26} and Sheffer \cite{27} interested in the same
problem: what is the all possible forms of polynomial sets which are at the
same time orthogonal and Sheffer polynomials. They stated that $A\left(
t\right) $ and $H\left( t\right) $ of $\left( \ref{3}\right) $ should
satisfy the following conditions
\begin{eqnarray*}
\frac{1}{H^{^{\prime }}\left( t\right) } &=&\left( 1-\alpha t\right) \left(
1-\beta t\right) \text{ \ \ \ ,} \\
\frac{A^{^{\prime }}\left( t\right) }{A\left( t\right) } &=&\frac{\lambda
_{2}t}{\left( 1-\alpha t\right) \left( 1-\beta t\right) }\text{ \ \ .}
\end{eqnarray*}%
If we discuss all possible cases in view of these two conditions, then we
face with the known orthogonal polynomial sets listed below:

\textbf{Case 1: }$\alpha =\beta =0\Rightarrow $ Hermite polynomials.

\textbf{Case 2: }$\alpha =\beta \neq 0\Rightarrow $ Laguerre polynomials.

\textbf{Case 3: }$\alpha \neq 0$, $\beta =0\Rightarrow $ Charlier
polynomials.

\textbf{Case 4: }$\alpha \neq \beta $, $\left( \alpha \text{, }\beta \in
%TCIMACRO{\U{211d} }%
%BeginExpansion
\mathbb{R}
%EndExpansion
\right) \Rightarrow $ Meixner polynomials.

\textbf{Case 5: }$\alpha \neq \beta $, $\left( \text{complex conjugate of
each other}\right) \Rightarrow $ Meixner-Pollaczek polynomials.

For more information see \cite{28}. Similar investigation was made in \cite%
{29} for $2$-orthogonal polynomials. In order to solve a characterization
problem for $d$-orthogonality, we need the following lemma.

\begin{lemma}
$\left( \text{\cite{14}}\right) $ Let $\left\{ P_{n}\right\} _{n\geq 0}$ be
a polynomial set defined by the following relation
\begin{equation*}
G\left( x,t\right) =A\left( t\right) G_{0}\left( x,t\right)
=\sum\limits_{k=0}^{\infty }P_{n}\left( x\right) \frac{t^{n}}{n!}
\end{equation*}%
with $G_{0}\left( 0,t\right) =1$ and let $\hat{X}_{t}$ and $\sigma :=\hat{T}%
_{x}$ be the transform operator of the multiplication operator by $x$ and $t$%
, respectively. Thus,
\begin{equation}
\left\{
\begin{array}{c}
\hat{X}_{t}G\left( x,t\right) =xG\left( x,t\right) \text{ \ ,} \\
\hat{T}_{x}G\left( x,t\right) =tG\left( x,t\right) \text{ \ .}%
\end{array}%
\right.  \label{4}
\end{equation}%
Then, $\left\{ P_{n}\right\} _{n\geq 0}$ is a $d$-orthogonal polynomial set
if and only if $\hat{X}_{t}\in \vee _{d+2}^{\left( -1\right) }$ where the
action of the set of operators $\tau \in \vee _{r}^{\left( -1\right) }$, $%
\left( r\geq 2\right) $, to $t^{n}$ is%
\begin{equation*}
\left\{
\begin{array}{cc}
\tau \left( 1\right) =\sum\limits_{k=1}^{r-1}\alpha _{k-1}^{\left( k\right)
}t^{k-1}\text{ \ ,} &  \\
\tau \left( t^{n}\right) =\sum\limits_{k=0}^{r-1}\alpha _{n}^{\left(
k\right) }t^{n+k-1}\text{ \ ,} & n\geq 1\text{ \ .}%
\end{array}%
\right.
\end{equation*}%
Here, $r$ complex sequences $\left\{ \alpha _{n}^{\left( k\right) }\right\}
_{n\geq 0}$ appear for $k=0,1,...,r-1$ with the condition $\alpha
_{n}^{\left( 0\right) }\alpha _{n}^{\left( r-1\right) }\neq 0$. Moreover,
the $d$-dimensional functional vector which ensures the $d$-orthogonality is
given by
\begin{equation}
\left\langle u_{i},f\right\rangle =\frac{1}{i!}\left[ \frac{\sigma ^{i}}{%
A\left( \sigma \right) }f\left( x\right) \right] _{x=0}\text{ \ , \ \ }%
i=0,1,...,d-1\text{ \ , \ }f\in \mathcal{P}\text{ \ .}  \label{5}
\end{equation}
\end{lemma}

From $\left( \ref{4}\right) $, It is obvious that $\sigma :=\hat{T}_{x}$ is
the lowering operator of $\left\{ P_{n}\right\} _{n\geq 0}$. Now, we can
state our main theorem.

\begin{theorem}
Let $\gamma _{d}\left( t\right) =\sum\limits_{k=0}^{d}\beta _{k}t^{k}$ be a
polynomial of degree $d$ $\left( \beta _{d}\neq 0\right) $ and $\sigma
_{d+1}\left( t\right) =\sum\limits_{k=0}^{d+1}\alpha _{k}t^{k}$ be a
polynomial of degree less than or equal to $d+1$. The only polynomial sets $%
\left\{ P_{n}\right\} _{n\geq 0}$ ,which are $d$-orthogonal and also Sheffer
polynomial set, are generated by
\begin{equation}
\exp \left[ \int\limits_{0}^{t}\frac{\gamma _{d}\left( s\right) }{\sigma
_{d+1}\left( s\right) }ds\right] \exp \left[ x\int\limits_{0}^{t}\frac{1}{%
\sigma _{d+1}\left( s\right) }ds\right] =\sum\limits_{n=0}^{\infty
}P_{n}\left( x\right) \frac{t^{n}}{n!}  \label{6}
\end{equation}%
with the conditions%
\begin{equation}
\alpha _{0}\left( n\alpha _{d+1}-\beta _{d}\right) \neq 0\text{ \ },\text{ \
\ }n\geq 1\text{ \ \ .}  \label{8}
\end{equation}
\end{theorem}

\begin{proof}
Let $\left\{ P_{n}\right\} _{n\geq 0}$ be a Sheffer polynomial set defined
by the generating function $\left( \ref{3}\right) $. Taking the derivative
of the both sides of the following equality
\begin{equation*}
G\left( x,t\right) =A\left( t\right) e^{xH\left( t\right) }
\end{equation*}%
with respect to $t$ leads to%
\begin{equation*}
\left[ \frac{1}{H^{^{\prime }}\left( t\right) }D_{t}-\frac{A^{^{\prime
}}\left( t\right) }{A\left( t\right) H^{^{\prime }}\left( t\right) }\right]
G\left( x,t\right) =xG\left( x,t\right)
\end{equation*}%
where $D_{t}=\frac{d}{dt}$. Thus, from $\left( \ref{4}\right) $%
\begin{equation*}
\hat{X}_{t}=\frac{1}{H^{^{\prime }}\left( t\right) }D_{t}-\frac{A^{^{\prime
}}\left( t\right) }{A\left( t\right) H^{^{\prime }}\left( t\right) }\text{ \
.}
\end{equation*}%
If $\left\{ P_{n}\right\} _{n\geq 0}$ is a $d$-orthogonal polynomial set,
according to Lemma 2.1, $\hat{X}_{t}$ should belong to the set of operators $%
\vee _{d+2}^{\left( -1\right) }$. This means that following equalities must
be satisfied%
\begin{equation}
\left\{
\begin{array}{c}
\frac{1}{H^{^{\prime }}\left( t\right) }=\sum\limits_{k=0}^{d+1}\alpha
_{k}t^{k}=\sigma _{d+1}\left( t\right)  \\
\frac{A^{^{\prime }}\left( t\right) }{A\left( t\right) H^{^{\prime }}\left(
t\right) }=\sum\limits_{k=0}^{d}\beta _{k}t^{k}=\gamma _{d}\left( t\right)
\text{ \ , \ \ }\beta _{d}\neq 0\text{ \ ,}%
\end{array}%
\right.   \label{7}
\end{equation}%
with the conditions $\left( \ref{8}\right) $. Solving equations $\left( \ref%
{7}\right) $ allows us to get the desired result given by $\left( \ref{6}%
\right) $.

Conversely, Let $\left\{ P_{n}\right\} _{n\geq 0}$ be a Sheffer polynomial
set generated by $\left( \ref{6}\right) $ with the conditions $\left( \ref{8}%
\right) $. Thus,%
\begin{equation*}
G\left( x,t\right) =\exp \left[ \int\limits_{0}^{t}\frac{\gamma _{d}\left(
s\right) }{\sigma _{d+1}\left( s\right) }ds\right] \exp \left[
x\int\limits_{0}^{t}\frac{1}{\sigma _{d+1}\left( s\right) }ds\right] \text{
\ \ .}
\end{equation*}%
If we apply derivative operator $D_{t}$ to the both sides of the above
equality, then we obtain%
\begin{equation*}
\left[ \sigma _{d+1}\left( t\right) D_{t}-\gamma _{d}\left( t\right) \right]
G\left( x,t\right) =xG\left( x,t\right) \text{ \ \ .}
\end{equation*}%
In view of Lemma 2.1 and the conditions $\left( \ref{8}\right) $
\begin{equation*}
\hat{X}_{t}=\sigma _{d+1}\left( t\right) D_{t}-\gamma _{d}\left( t\right)
\in \vee _{d+2}^{\left( -1\right) }\text{ \ ,}
\end{equation*}%
so $\left\{ P_{n}\right\} _{n\geq 0}$ is a $d$-orthogonal polynomial set.
\end{proof}

\begin{remark}
This characterization of $d$-orthogonal Sheffer polynomial sets seems to be
new in this notion. For $d=1$, these results reduce to the ones obtained for
orthogonal polynomials which are summarized in the beginning of this section.
\end{remark}

Next, as applications of Theorem 2.2, we revisit some known $d$-orthogonal
polynomial sets which are at the same time Sheffer polynomial sets. Also, we
derive new $d$-orthogonal polynomial sets and find their $d$-dimensional
functional vector.

$\left( i\right) $ \textbf{Laguerre type }$d$\textbf{-orthogonal polynomial
sets}

During the last decade, authors have paid so much attention to extend
Laguerre polynomials to $d$-orthogonality. Thus, there are several
extensions of Laguerre polynomials in the context of $d$-orthogonality$%
\left( \text{see \cite{7}, \cite{14} and \cite{22}}\right) $. Now, taking
Theorem 2.2 into account, we revisit some of them which are explicitly
obtained from the couple of polynomials $\left[ \gamma _{d}\left( t\right)
\text{, }\sigma _{d+1}\left( t\right) \right] $.

\textbf{Application 1: }Let $\left\{ P_{n}\right\} _{n\geq 0}$ be a Sheffer
polynomial set generated by $\left( \ref{6}\right) $ with the following
couple of polynomials
\begin{equation*}
\left[ \gamma _{d}\left( t\right) \text{, }\sigma _{d+1}\left( t\right) %
\right] =\left[ -\left( \alpha +1\right) \left( 1-t\right) ^{d}\text{, }%
\frac{-1}{d}\left( 1-t\right) ^{d+1}\right] \text{ \ \ .}
\end{equation*}%
where $\alpha \neq -1$. After some calculations, thanks to Theorem 2.2, we
obtain a $d$-orthogonal polynomial set of the form
\begin{equation}
\left( 1-t\right) ^{-\left( \alpha +1\right) d}\exp \left\{ -x\left[ \left(
1-t\right) ^{-d}-1\right] \right\} =\sum\limits_{n=0}^{\infty }P_{n}\left(
x\right) \frac{t^{n}}{n!}  \label{9}
\end{equation}%
with the conditions $\frac{n}{d}+\alpha +1\neq 0$. The $d$-orthogonality of
this polynomial set deeply investigated in \cite{22}. Also, the authors
stated basic properties of these polynomials.

\textbf{Application 2: }Assume that $\left\{ P_{n}\right\} _{n\geq 0}$ is a
Sheffer polynomial set which has the generating function of the form $\left( %
\ref{6}\right) $ due to the couple of polynomials given below%
\begin{equation*}
\left[ \gamma _{d}\left( t\right) \text{, }\sigma _{2}\left( t\right) \right]
=\left[ -\left( 1-t\right) ^{2}\pi _{d-1}^{^{\prime }}\left( t\right)
-\left( \alpha +1\right) \left( 1-t\right) \text{, }-\left( 1-t\right) ^{2}%
\right] \text{ \ \ .}
\end{equation*}%
Here $\pi _{d-1}\left( t\right) =\sum\limits_{k=0}^{d-1}a_{k}t^{k}$ with $%
a_{d-1}\neq 0$. According to Theorem 2.2, necessary computations lead us to
get $d$-orthogonal polynomial sets of the type%
\begin{equation}
e^{\pi _{d-1}\left( t\right) }\left( 1-t\right) ^{-\alpha -1}\exp \left(
\frac{-xt}{1-t}\right) =\sum\limits_{n=0}^{\infty }P_{n}\left( x\right)
\frac{t^{n}}{n!}\text{ \ \ .}  \label{10}
\end{equation}%
$a_{d-1}\neq 0$ is enough to satisfy the conditions $\left( \ref{8}\right) $%
. These $d$-orthogonal polynomial sets found and studied in \cite{14}.

\begin{remark}
These two polynomial sets $\left( \ref{9}\right) $ and $\left( \ref{10}%
\right) $ are the generalizations of Laguerre polynomials to the $d$%
-orthogonal polynomials since we meet Laguerre polynomials for $d=1$.
\end{remark}

Now, we present a new $d$-orthogonal polynomial set for $d\geq 2$. It seems
that this polynomial set is a Laguerre type $d$-orthogonal polynomial set
but the difference is Laguerre polynomials are not generated hence $d\geq 2$.

\textbf{Application 3: }Suppose that $\left\{ P_{n}\right\} _{n\geq 0}$ is a
Sheffer polynomial set possessing the generating function $\left( \ref{6}%
\right) $ relate to the couple of polynomials
\begin{equation*}
\left[ \gamma _{d}\left( t\right) \text{, }\sigma _{3}\left( t\right) \right]
=\left[ -\left( 1-t\right) ^{3}\pi _{d-2}^{^{\prime }}\left( t\right)
-\left( \alpha +1\right) \left( 1-t\right) ^{2}\text{, }-\left( 1-t\right)
^{3}\right] \text{ \ \ }
\end{equation*}%
where $\pi _{d-2}\left( t\right) =\sum\limits_{k=0}^{d-2}a_{k}t^{k}$ with $%
a_{d-2}\neq 0$. Then, taking Theorem 2.2 into account, we have a new $d$%
-orthogonal polynomial set generated by
\begin{equation*}
e^{\pi _{d-2}\left( t\right) }\left( 1-t\right) ^{-\alpha -1}\exp \left( x%
\frac{t^{2}-2t}{2\left( 1-t\right) ^{2}}\right) =\sum\limits_{n=0}^{\infty
}P_{n}\left( x\right) \frac{t^{n}}{n!}\text{ \ \ }
\end{equation*}%
for $d\geq 2$. Similarly, $a_{d-2}\neq 0$ guarantees the conditions $\left( %
\ref{8}\right) $. Now, we deal with the case $d=2$ i.e.: a new $2$%
-orthogonal polynomial set. Thus%
\begin{equation}
\left( 1-t\right) ^{-\alpha -1}\exp \left( x\frac{t^{2}-2t}{2\left(
1-t\right) ^{2}}\right) =\sum\limits_{n=0}^{\infty }P_{n}\left( x\right)
\frac{t^{n}}{n!}\text{ \ }  \label{11}
\end{equation}%
with conditions $\alpha +n+1\neq 0$ , $n\geq 0$. Before finding the
corresponding linear functionals $u_{0}$ and $u_{1}$ of this new $2$%
-orthogonal polynomial set, we need to state following useful lemma.

\begin{lemma}
$\left( \text{\cite{30}}\right) $ Let $A\left( t\right) $ and $H\left(
t\right) $ be two power series given as in $\left( \ref{3}\right) $ and
\begin{equation*}
H^{\ast }\left( t\right) =\sum\limits_{k=0}^{\infty }h_{k}^{\ast }t^{k+1}
\end{equation*}%
is the compositional inverse of $H\left( t\right) $ such that%
\begin{equation*}
H\left( H^{\ast }\left( t\right) \right) =H^{\ast }\left( H\left( t\right)
\right) =t\text{ \ .}
\end{equation*}%
$\left( i\right) $ The lowering operator $\sigma :=\hat{T}_{x}$ of the
polynomial set $\left\{ P_{n}\right\} _{n\geq 0}$ generated by $\left( \ref%
{3}\right) $ is given with
\begin{equation*}
\sigma =H^{\ast }\left( D\right) \text{ \ \ , \ \ }D=\frac{d}{dx}\text{ \ .}
\end{equation*}%
$\left( ii\right) $ The lowering operator $\sigma :=\hat{T}_{x}$ of the
polynomial set $\left\{ P_{n}\right\} _{n\geq 0}$ generated by%
\begin{equation*}
A\left( t\right) \left( 1+\omega H\left( t\right) \right) ^{\frac{x}{\omega }%
}=\sum\limits_{n=0}^{\infty }P_{n}\left( x\right) \frac{t^{n}}{n!}
\end{equation*}%
is given by
\begin{equation*}
\sigma =H^{\ast }\left( \Delta _{\omega }\right) \text{ \ \ , \ \ \ }\Delta
_{\omega }\left[ f\left( x\right) \right] =\frac{f\left( x+\omega \right)
-f\left( x\right) }{\omega }\text{.}
\end{equation*}
\end{lemma}

\begin{theorem}
The polynomial set $\left\{ P_{n}\right\} _{n\geq 0}$ generated by $\left( %
\ref{11}\right) $ are $2$-orthogonal for $\alpha >-1$ with respect to the
following linear functionals%
\begin{eqnarray*}
\left\langle u_{0},f\right\rangle &=&\int\limits_{0}^{\infty }\Psi _{\alpha
,f}\left( x\right) e^{-x}dx\text{ \ \ \ , \ \ }f\in \mathcal{P}\text{ \ ,} \\
\left\langle u_{1},f\right\rangle &=&\int\limits_{0}^{\infty }\left[ \Psi
_{\alpha ,f}\left( x\right) -\Psi _{\alpha +1,f}\left( x\right) \right]
e^{-x}dx\text{ \ \ \ , \ \ }f\in \mathcal{P}\text{ \ ,}
\end{eqnarray*}%
where
\begin{equation*}
\Psi _{\alpha ,f}\left( x\right) =\frac{x^{\alpha }}{\Gamma \left( \alpha
+1\right) }\sum\limits_{k=0}^{\infty }\frac{f^{\left( k\right) }\left(
0\right) }{k!\left( \frac{\alpha +2}{2}\right) _{k}}\left( \frac{x^{2}}{2}%
\right) ^{k}\text{ \ \ \ ,}
\end{equation*}%
$\Gamma $ is the widely known Gamma function and $\left( a\right) _{n}$ is
the Pochhammer's symbol defined by the rising factorial%
\begin{eqnarray*}
\left( a\right) _{n} &=&a\left( a+1\right) ...\left( a+n-1\right) \text{ \ \
, \ \ }n\geq 1\text{ \ ,} \\
\left( a\right) _{0} &=&1\text{ \ \ .}
\end{eqnarray*}
\end{theorem}

\begin{proof}
$\left( \ref{5}\right) $ yields that
\begin{equation*}
\left\langle u_{i},f\right\rangle =\frac{1}{i!}\left[ \frac{\sigma ^{i}}{%
A\left( \sigma \right) }f\left( x\right) \right] _{x=0}\text{ \ , \ \ }i=0,1%
\text{ \ , \ }f\in \mathcal{P}\text{ \ ,}
\end{equation*}%
where $\sigma $ is the lowering operator of $2$-orthogonal polynomial set
generated by $\left( \ref{11}\right) $. The lowering operator $\sigma $ of
this polynomial set is
\begin{equation*}
H\left( t\right) =\frac{t^{2}-2t}{2\left( 1-t\right) ^{2}}\Rightarrow \sigma
=H^{\ast }\left( D\right) =1-\left( 1-2D_{x}\right) ^{-1/2}
\end{equation*}%
where we use Lemma 2.5. Then, for $i=0$ and $A\left( t\right) =\left(
1-t\right) ^{-\alpha -1}$, we obtain%
\begin{eqnarray*}
\left\langle u_{0},f\right\rangle &=&\left[ \left( 1-2D_{x}\right) ^{-\left(
\frac{\alpha +1}{2}\right) }f\left( x\right) \right] _{x=0} \\
&=&\sum\limits_{k=0}^{\infty }\frac{\left( \frac{\alpha +1}{2}\right)
_{k}2^{k}}{k!}f^{\left( k\right) }\left( 0\right) \\
&=&\sum\limits_{k=0}^{\infty }\frac{\Gamma \left( \alpha +2k+1\right) }{%
\Gamma \left( \alpha +1\right) \left( \frac{\alpha +2}{2}\right) _{k}2^{k}}%
\frac{f^{\left( k\right) }\left( 0\right) }{k!} \\
&=&\int\limits_{0}^{\infty }\left[ \frac{x^{\alpha }}{\Gamma \left( \alpha
+1\right) }\sum\limits_{k=0}^{\infty }\frac{f^{\left( k\right) }\left(
0\right) }{k!\left( \frac{\alpha +2}{2}\right) _{k}}\left( \frac{x^{2}}{2}%
\right) ^{k}\right] e^{-x}dx\text{ \ \ \ \ .}
\end{eqnarray*}%
Furthermore, we calculate in a similar manner for $i=1$%
\begin{eqnarray*}
\left\langle u_{1},f\right\rangle &=&\left[ \sum\limits_{r=0}^{1}\binom{1}{r%
}\left( -1\right) ^{r}\left( 1-2D_{x}\right) ^{-\left( \frac{\alpha +r+1}{2}%
\right) }f\left( x\right) \right] _{x=0} \\
&=&\sum\limits_{r=0}^{1}\binom{1}{r}\left( -1\right)
^{r}\sum\limits_{k=0}^{\infty }\frac{\left( \frac{\alpha +r+1}{2}\right)
_{k}2^{k}}{k!}f^{\left( k\right) }\left( 0\right) \\
&=&\sum\limits_{r=0}^{1}\binom{1}{r}\left( -1\right)
^{r}\sum\limits_{k=0}^{\infty }\frac{\Gamma \left( \alpha +r+2k+1\right) }{%
\Gamma \left( \alpha +r+1\right) \left( \frac{\alpha +r+2}{2}\right)
_{k}2^{k}}\frac{f^{\left( k\right) }\left( 0\right) }{k!} \\
&=&\int\limits_{0}^{\infty }\left[ \sum\limits_{r=0}^{1}\binom{1}{r}\left(
-1\right) ^{r}\frac{x^{\alpha +r}}{\Gamma \left( \alpha +r+1\right) }%
\sum\limits_{k=0}^{\infty }\frac{f^{\left( k\right) }\left( 0\right) }{%
k!\left( \frac{\alpha +r+2}{2}\right) _{k}}\left( \frac{x^{2}}{2}\right) ^{k}%
\right] e^{-x}dx\text{ \ .}
\end{eqnarray*}%
This finishes the proof.
\end{proof}

$\left( ii\right) $ \textbf{Hermite type }$d$\textbf{-orthogonal polynomial
sets}

Hermite type $d$-orthogonal polynomials were the first example of $d$%
-orthogonal polynomials which obtained constructively by Douak \cite{16}. He
discovered these polynomials as solution of the problem: Find all $d$%
-orthogonal polynomial sets which are at the same time Appell polynomials.

\textbf{Application 4: }Let $\left\{ P_{n}\right\} _{n\geq 0}$ be a Sheffer
polynomial set due to the generating function $\left( \ref{6}\right) $ and
the following couple of polynomials
\begin{equation*}
\left[ \gamma _{d}\left( t\right) \text{, }\sigma _{0}\left( t\right) \right]
=\left[ \pi _{d+1}^{^{\prime }}\left( t\right) \text{, }1\right]
\end{equation*}%
where $\pi _{d+1}\left( t\right) =\sum\limits_{k=0}^{d+1}a_{k}t^{k}$ with $%
a_{d+1}\neq 0$. Theorem 2.2 allows us to present a $d$-orthogonal set
generated by
\begin{equation}
e^{\pi _{d+1}\left( t\right) }\exp \left( xt\right)
=\sum\limits_{n=0}^{\infty }P_{n}\left( x\right) \frac{t^{n}}{n!}\text{ \ \
}  \label{12}
\end{equation}%
under the conditions $a_{d+1}\neq 0$ from $\left( \ref{8}\right) $. The only
polynomials which are $d$-orthogonal and Appell polynomials at the same time
are generated by $\left( \ref{12}\right) $. Also, these polynomial sets are
the generalization of Gould-Hopper polynomials \cite{32}. For $d=1$, we meet
again the Hermite polynomials. The properties of these polynomial sets were
intensively studied by Douak in \cite{16}.

$\left( iii\right) $ \textbf{Charlier type }$d$\textbf{-orthogonal
polynomial sets}

Also, Charlier polynomials extended to the notion of $d$-orthogonality. Now,
we revisit these polynomial sets owing to Theorem 2.2.

\textbf{Application 5: }Suppose that $\left\{ P_{n}\right\} _{n\geq 0}$ is a
Sheffer polynomial set generated by $\left( \ref{6}\right) $ associated with
the couple of polynomials
\begin{equation*}
\left[ \gamma _{d}\left( t\right) \text{, }\sigma _{1}\left( t\right) \right]
=\left[ \left( 1+\omega t\right) \pi _{d}^{^{\prime }}\left( t\right) \text{%
, }1+\omega t\right]
\end{equation*}%
where $\pi _{d}\left( t\right) =\sum\limits_{k=0}^{d}a_{k}t^{k}$ with $%
a_{d}\neq 0$. From Theorem 2.2, we have a $d$-orthogonal polynomial set of the
form
\begin{equation}
e^{\pi _{d}\left( t\right) }\left( 1+\omega t\right) ^{\frac{x}{\omega }%
}=\sum\limits_{n=0}^{\infty }P_{n}\left( x\right) \frac{t^{n}}{n!}\text{ \
\ }  \label{13}
\end{equation}%
with the conditions $a_{d}\neq 0$ from $\left( \ref{8}\right) $. The $d$%
-orthogonal polynomial set generated by $\left( \ref{13}\right) $ was found
in \cite{13}. A similar characterization problem stated in $\left( ii\right)
$ for discrete case was solved by the authors. It is obvious that the
generating function relation $\left( \ref{13}\right) $ yields Charlier
polynomials for $\left( d,\omega \right) =\left( 1,1\right) $.

$\left( iii\right) $ \textbf{Meixner type }$d$\textbf{-orthogonal polynomial
sets}

Another important member of discrete orthogonal polynomial sets called
Meixner polynomials were also generalized in the context of $d$%
-orthogonality \cite{14}. The authors discovered these polynomials by means
of a form of generating function and they found the linear functions $u_{0}$
and $u_{1}$ for the case $d=2$. Following application of Theorem 2.2 shows
that these polynomial sets can be obtained by the special case of the couple
of polynomials $\left[ \gamma _{d}\left( t\right) \text{, }\sigma
_{d+1}\left( t\right) \right] $.

\textbf{Application 6: }Assume that $\left\{ P_{n}\right\} _{n\geq 0}$ is a
Sheffer polynomial set having the generating function of the type $\left( %
\ref{6}\right) $ for the couple of polynomials
\begin{equation*}
\left[ \gamma _{d}\left( t\right) \text{, }\sigma _{2}\left( t\right) \right]
=\left[ \frac{1}{c-1}\left( c-t\right) \left( 1-t\right) \pi
_{d-1}^{^{\prime }}\left( t\right) +\frac{\beta }{c-1}\left( c-t\right)
\text{, }\frac{1}{c-1}\left( c-t\right) \left( 1-t\right) \right] \text{ .}
\end{equation*}%
Here $\pi _{d-1}\left( t\right) =\sum\limits_{k=0}^{d-1}a_{k}t^{k}$ with $%
a_{d-1}\neq 0$ and $c\neq \left\{ 0,1\right\} $. Thanks to Theorem 2.2 and
this couple of polynomials, $\left( \ref{6}\right) $ generates the $d$%
-orthogonal polynomial sets given below%
\begin{equation}
e^{\pi _{d-1}\left( t\right) }\left( 1-t\right) ^{-\beta }\left( 1+\frac{c-1%
}{c}\frac{t}{1-t}\right) ^{x}=\sum\limits_{n=0}^{\infty }P_{n}\left(
x\right) \frac{t^{n}}{n!}\ \ \text{.}  \label{14}
\end{equation}%
$a_{d-1}\neq 0$ and $c\neq \left\{ 0,1\right\} $ are sufficient enough the
conditions $\left( \ref{8}\right) $ hold true. One can find detailed
information of these polynomial sets in \cite{14}. It is easily seen that we
face with the Meixner polynomial set by taking $d=1$ in $\left( \ref{14}%
\right) $. Recently, a generalization of $d$-orthogonal Meixner polynomial
sets via quantum calculus has been given in \cite{31}. Next, we express a
new Meixner type $d$-orthogonal polynomial set and we \ find its $d$%
-dimensional functional vector.

\textbf{Application 7: }Let\textbf{\ }$\left\{ P_{n}\right\} _{n\geq 0}$ be
a Sheffer polynomial set generated by $\left( \ref{6}\right) $ according to
the couple of polynomials
\begin{eqnarray*}
\left[ \gamma _{d}\left( t\right) \text{, }\sigma _{d+1}\left( t\right) %
\right]  &=&\left[ \frac{dc\beta }{c-1}\left[ \left( 1-t\right) ^{d}+\frac{%
c-1}{dc}\left[ 1-\left( 1-t\right) ^{d}\right] \right] ,\right.  \\
&&\left. \frac{c}{c-1}\left( 1-t\right) \left[ \left( 1-t\right) ^{d}+\frac{%
c-1}{dc}\left[ 1-\left( 1-t\right) ^{d}\right] \right] \right]
\end{eqnarray*}%
with the restrictions
\begin{equation}
\left\{
\begin{array}{c}
c\neq \left\{ 0,\frac{1}{1-d},1\right\}  \\
\beta \neq -\frac{n}{d}\text{ \ , \ }n\geq 0\text{ }%
\end{array}%
\right. \text{ \ .}  \label{15}
\end{equation}%
After some computations, Theorem 2.2 allows us to introduce the following new $%
d$-orthogonal polynomial set with
\begin{equation}
\left( 1-t\right) ^{-\beta d}\left( 1+\frac{c-1}{dc}\left[ \frac{1}{\left(
1-t\right) ^{d}}-1\right] \right) ^{x}=\sum\limits_{n=0}^{\infty
}P_{n}\left( x\right) \frac{t^{n}}{n!}\ \ \text{.}  \label{16}
\end{equation}%
The conditions $\left( \ref{8}\right) $ are satisfied from restrictions $%
\left( \ref{15}\right) $. It seems that this Meixner type $d$-orthogonal
polynomial set is the first explicit one among others in the literature.

\begin{theorem}
The $d$-dimensional functional vectors, which the $d$-orthogonality of the
polynomial set generated by $\left( \ref{16}\right) $ holds, are
\begin{equation}
\left\langle u_{r},f\right\rangle =\frac{1}{r!}\sum\limits_{i=0}^{r}\binom{r%
}{i}\left( -1\right) ^{i}\sum\limits_{j=0}^{\infty }\frac{\left( \beta +%
\frac{i}{d}\right) _{j}\left( \frac{dc}{1-c}\right) ^{j}}{\left( 1-\frac{dc}{%
c-1}\right) ^{\beta +\frac{i}{d}+j}j!}f\left( j\right)  \label{19}
\end{equation}%
where $r=0,1,...,d-1$ and $f\in \mathcal{P}$.
\end{theorem}

\begin{proof}
Lemma 2.5 helps us to find the lowering operators of the $d$-orthogonal
polynomial set generated by $\left( \ref{16}\right) $ with
\begin{equation*}
H\left( t\right) =\frac{c-1}{dc}\left[ \frac{1}{\left( 1-t\right) ^{d}}-1%
\right] \Rightarrow \sigma =H^{\ast }\left( \Delta \right) =1-\left( 1-\frac{%
dc\Delta }{1-c}\right) ^{-\frac{1}{d}}
\end{equation*}%
where $\Delta f\left( x\right) =f\left( x+1\right) -f\left( x\right) $.
Thus, by using Lemma 2.1 and $\left( \ref{5}\right) $, we conclude that for $%
r=0,1,...,d-1$ and $f\in \mathcal{P}$%
\begin{eqnarray}
\left\langle u_{r},f\right\rangle &=&\frac{1}{r!}\left[ \frac{\sigma ^{r}}{%
A\left( \sigma \right) }f\left( x\right) \right] _{x=0}  \notag \\
&=&\frac{1}{r!}\left[ \sum\limits_{i=0}^{r}\binom{r}{i}\left( -1\right)
^{i}\left( 1-\frac{dc\Delta }{1-c}\right) ^{-\left( \beta +\frac{i}{d}%
\right) }f\left( x\right) \right] _{x=0}  \notag \\
&=&\frac{1}{r!}\sum\limits_{i=0}^{r}\binom{r}{i}\left( -1\right)
^{i}\sum\limits_{k=0}^{\infty }\frac{\left( \beta +\frac{i}{d}\right) _{k}}{%
k!}\left( \frac{dc}{1-c}\right) ^{k}\Delta ^{k}f\left( 0\right) \text{ \ \ .}
\label{17}
\end{eqnarray}%
Substituting the fact
\begin{equation*}
\Delta ^{k}f\left( 0\right) =\sum\limits_{j=0}^{k}\left( -1\right) ^{k-j}%
\binom{k}{j}f\left( j\right)
\end{equation*}%
into $\left( \ref{17}\right) $ and after shifting indices, we obtain%
\begin{equation}
\left\langle u_{r},f\right\rangle =\frac{1}{r!}\sum\limits_{i=0}^{r}\binom{r%
}{i}\left( -1\right) ^{i}\sum\limits_{j=0}^{\infty }\left\{
\sum\limits_{k=0}^{\infty }\frac{\left( \beta +\frac{i}{d}\right)
_{k+j}\left( -1\right) ^{k}\left( \frac{dc}{1-c}\right) ^{k}}{k!}\right\}
\frac{\left( \frac{dc}{1-c}\right) ^{j}f\left( j\right) }{j!}\text{ .}
\label{18}
\end{equation}%
The equality $\left( \ref{18}\right) $ leads us to get the desired result by
applying the following property of the Pochhammer's symbol%
\begin{equation*}
\left( \beta +\frac{i}{d}\right) _{k+j}=\left( \beta +\frac{i}{d}\right)
_{j}\left( \beta +\frac{i}{d}+j\right) _{k}\text{ \ \ .}
\end{equation*}
\end{proof}

\begin{remark}
For $d=1$, $\left( \ref{16}\right) $ reduces to the well known generating
function of Meixner polynomial set and $\left( \ref{19}\right) $ becomes the
following linear functional for Meixner polynomials
\begin{equation}
\left\langle u_{0},f\right\rangle =\left( 1-c\right) ^{\beta
}\sum\limits_{j=0}^{\infty }\frac{\left( \beta \right) _{j}\text{ }c^{j}}{j!%
}f\left( j\right)  \label{20}
\end{equation}%
with $0<c<1$ and $\beta >0$. Meixner polynomial set is orthogonal with
respect to the linear functional given by $\left( \ref{20}\right) $.
\end{remark}

\textbf{Application 8: }Suppose that $\left\{ P_{n}\right\} _{n\geq 0}$ is a
Sheffer polynomial set represented by $\left( \ref{6}\right) $ associated to
the couple of polynomials
\begin{eqnarray*}
\left[ \gamma _{d}\left( t\right) \text{, }\sigma _{3}\left( t\right) \right]
&=&\left[ -\frac{c}{c-1}\left( 1-t\right) \left[ \left( 1-t\right) ^{2}+%
\frac{c-1}{2c}\left[ \left( 1-t\right) ^{2}-1\right] \right] \pi
_{d-2}^{^{\prime }}\left( t\right) \right. \\
&&\left. -\frac{\beta c}{c-1}\left[ \left( 1-t\right) ^{2}+\frac{c-1}{2c}%
\left[ \left( 1-t\right) ^{2}-1\right] \right] ,\right. \\
&&\left. -\frac{c}{c-1}\left( 1-t\right) \left[ \left( 1-t\right) ^{2}+\frac{%
c-1}{2c}\left[ \left( 1-t\right) ^{2}-1\right] \right] \right]
\end{eqnarray*}%
where $\pi _{d-2}\left( t\right) =\sum\limits_{k=0}^{d-2}a_{k}t^{k}$ with $%
a_{d-2}\neq 0$ and $c\neq \left\{ 0,\frac{1}{3},1\right\} $. Taking Theorem
2.2 into account, we derive a new $d$-orthogonal polynomial set for $d\geq 2$
with%
\begin{equation}
e^{\pi _{d-2}\left( t\right) }\left( 1-t\right) ^{-\beta }\left( 1+\frac{c-1%
}{2c}\frac{t^{2}-2t}{\left( 1-t\right) ^{2}}\right)
^{x}=\sum\limits_{n=0}^{\infty }P_{n}\left( x\right) \frac{t^{n}}{n!}
\label{21}
\end{equation}%
The conditions $\left( \ref{8}\right) $ are satisfied since $a_{d-2}\neq 0$
and $c\neq \left\{ 0,\frac{1}{3},1\right\} $. These $d$-orthogonal
polynomial sets $\left( \ref{21}\right) $ can not generate an orthogonal
polynomial set since $d\geq 2$. But for $d=2$, one can study the properties
of this Meixner type $2$-orthogonal polynomial set.

\subsection{Concluding Remarks}
Theorem 2.2 is the generalization of the characterization problem related to
the orthogonality of Sheffer polynomial set. The version of this problem
corresponding to $d=1$ and $d=2$ already exist in the literature. Then, it
is expected to find similar results for $d$-orthogonality. Although the
results obtained in Theorem 2.2 are expected and natural, this theorem
motivates us to derive new $d$-orthogonal polynomial sets as mentioned in
this paper. One can generate more $d$-orthogonal polynomial sets which are
Sheffer polynomial set at the same time with the help of Theorem 2.2.

\end{document}